\documentclass[10pt,a4paper]{amsart}
\usepackage{amsmath,amssymb,amsthm}
\usepackage{mathtools}
\usepackage{hyperref}
\usepackage{geometry}
\usepackage{enumitem}
\usepackage{bm}

\newtheorem{theorem}{Theorem}[section]
\newtheorem{proposition}[theorem]{Proposition}
\newtheorem{lemma}[theorem]{Lemma}
\newtheorem{corollary}[theorem]{Corollary}

\theoremstyle{definition}
\newtheorem{definition}[theorem]{Definition}

\newtheorem{remark}[theorem]{Remark}
\newtheorem{assumption}[theorem]{Assumption}

\newcommand{\C}{\mathbb{C}}

\newcommand{\Z}{\mathbb{Z}}
\newcommand{\D}{\mathbb{D}}

\newcommand{\Vx}[1]{V_{\chi_{#1}}}
\newcommand{\VxS}[1]{V^{S_{#1}}_{\chi_{#1}}}
\newcommand{\VxX}[1]{V^{X_{#1}}_{\chi_{#1}}}
\newcommand{\Del}[1]{\Delta_{#1}}

\newcommand{\piX}[2]{\pi_1(#1,#2)}

\newcommand{\rk}{\mathrm{rank}}
\newcommand{\id}{\mathrm{id}}
\DeclareMathOperator{\Hom}{Hom}

\newcommand{\Bez}{\mathcal{B}}
\newcommand{\Ves}{\mathfrak{V}}

\newcommand{\hp}{\widetilde{p}}
\newcommand{\hX}{\widetilde{X}}

\begin{document}
\title[Hardy spaces on Riemann surfaces under ramified coverings]
      {Hardy spaces on Riemann surfaces under ramified coverings}
\author{Alexander Zuevsky}
\address{Institute of Mathematics, Czech Academy of Sciences,
         \v{Z}itn\'{a} 25, 115\;67 Prague 1, Czech Republic}
\email{zuevsky@yahoo.com}



\begin{abstract}
We extend the theory of indefinite Hardy spaces on finite bordered Riemann
surfaces to the setting of ramified analytic coverings.
Given a finite $n$-sheeted ramified covering $F\colon S_1\to S_2$ of
finite bordered Riemann surfaces satisfying a spin-compatibility hypothesis,
we construct
(i) the direct image of a unitary flat vector bundle $\VxX{1}\otimes\Del{1}$
on the double $X_1$ under $F$, taking full account of the 
ramification divisor $R_F$ and establishing the extension across the 
branch locus via a local analysis; 
(ii) a canonical matrix function $G_2$ encoding the parahermitian 
structure on $X_2$, together with the induced representation 
$\chi_2$ of $\piX{X_2}{p_0}$; 
(iii) an explicit isometric isomorphism
$\phi_F\colon H^{2,J_1(p)}(S_1,\VxS{1}\otimes\Del{1})
\xrightarrow{\;\sim\;} H^{2,J_2(p)}(S_2,\VxS{2}\otimes\Del{2})$ 
between the associated Hardy-Kre\u{\i}n spaces, provided 
that $h^0(X_1,\VxX{1}\otimes\Del{1})=0$ and that the branch locus
is disjoint from $\partial S_2$. 
We then develop the resulting operator theory in terms of 
vessels and Bezoutian operators.
To each object in the category $\mathcal{RH}$ of finite bordered surfaces
with unitary flat bundles we attach a triangular vessel whose input and
output spaces are the Hardy-Kre\u{\i}n spaces on the two surfaces.  
 The Bezoutian of the vessel is expressed as a finite-rank operator on $\mathcal{H}_2$
whose kernel is built from bounded holomorphic point-evaluation functionals in $\mathcal{H}_2$ 
evaluated at the interior ramification images $F(r_\nu)\in S_2$, 
consistently with the boundary-transversality hypothesis $\partial S_2\cap B_F=\left\{\varnothing\right\}$.
We prove that the assignment
$(S,\Vx{},J)\mapsto H^{2,J(p)}(S,\Vx{}\otimes\Delta)$ extends to a covariant
functor from $\mathcal{RH}$ (with ramified morphisms) to the category of
Kre\u{\i}n spaces.
\end{abstract}
\maketitle

\section{Introduction}
\label{secintro}
Hardy spaces on finite bordered Riemann surfaces with indefinite inner
products were introduced and systematically studied in
\cite{AlpayVinnikov2000}.
The main construction there replaces the usual
unit-disk model by the double $X$ of a bordered surface $S$, uses Cauchy
kernels for vector bundles on compact Riemann surfaces, and builds an
explicit isometric isomorphism to a Hardy-Kre\u{\i}n space over $\D$.
In \cite{Zuevsky2009, Zuevsky2015} it was shown that an  {unramified} covering
$F\colon S_1\to S_2$ induces a canonical isometric isomorphism between the
associated Hardy-Kre\u{\i}n spaces.
The present paper addresses the remaining and more subtle case of ramified coverings.
Ramification introduces zeroes of the derivative of $F$, equivalently, poles of the
reciprocal Jacobian, which invalidate the
change-of-variables argument used in the unramified case and necessitate
a correction at the ramification locus.
The central new difficulty compared to the unramified case is that the
half-order differential $\Delta_1$ satisfying $\Delta_1^{\otimes 2}\cong K_{X_1}$
and the compatibility $\Delta_1\cong F^*\Delta_2\otimes\mathcal{O}_{X_1}(\tfrac{1}{2}R_F)$
requires the existence of a  {spin structure} on $X_1$ compatible with
$F$ and $\Delta_2$. This is a cohomological obstruction
(see Section \ref{secdirectimage} and Assumption \ref{assspin} below).
We make this hypothesis explicit throughout and show that under it all
constructions are globally well-defined.
A second hypothesis (Assumption \ref{assboundary})
ensures that the branch locus is disjoint from $\partial S_2$. Without
it several formulas require modification.

\textbf{Overview of results.}  
Let $F\colon S_1\to S_2$ be a finite $n$-sheeted analytic map that is ramified. 
 We allow $F$ to extend to a ramified covering
$F\colon X_1\to X_2$ of the doubles equivariant with respect to the
anti-holomorphic involutions.
The ramification divisor on $X_1$ is $R_F=\sum_{p\in X_1}(e_p-1)p$    
where $e_p\geq 1$ denotes the ramification index of $F$ at $p$.

\textit{Spin structure and half-order differentials (Section \ref{secdirectimage}).}
The main geometric hypothesis (Assumption \ref{assspin}) is the existence
of an $F$-compatible spin structure: a line bundle $\mathcal{L}$ on $X_1$
with $\mathcal{L}^{\otimes 2}\cong K_{X_1}$ and $\mathcal{L}\cong F^*\Delta_2\otimes
\mathcal{O}_{X_1}(\tfrac{1}{2}R_F)$, together with a $\tau_1$-equivariant
trivialisation of the sign ambiguity.
We verify this hypothesis explicitly
in the examples and show it is equivalent to the existence of a theta
characteristic on $X_1$ compatible with the covering data.

\textit{Bundle-theoretic results (Section \ref{secdirectimage}).}
Under Assumption \ref{assspin}, we construct the  {ramified direct
image} $F_*^{\rm ram}(\VxX{1}\otimes\Del{1})$ on $X_2$, prove that it
extends holomorphically across the branch locus $B_F$ as a rank-$nm$ bundle,
and verify that it is isomorphic to $\VxX{2}\otimes\Del{2}$ as a coherent
sheaf on all of $X_2$.

\textit{Boundary hypothesis (Section \ref{secpairing}).}
We impose Assumption \ref{assboundary}: $\partial S_2\cap B_F=\left\{\varnothing\right\}$,
i.e., the branch values of $F$ lie in the interior $S_2$.
This ensures
that the inner product formula \eqref{eqip2} is non-singular on $\partial S_2$.

\textit{Main theorem (Section \ref{secmainthm}).}
Under Assumptions \ref{assspin} and \ref{assboundary}, and assuming
$h^0(X_1,\VxX{1}\otimes\Del{1})=0$, we prove that the map
$\phi_F\colon \hat{f}^1\longmapsto\hat{f}^2$ 
defined by the ramified analogue of formula (13) of \cite{Zuevsky2009}, 
incorporating the local Jacobian factors $|\varphi_i'|$ with their
ramification orders and the $\sqrt{e_i}$-normalisation, is an isometric
isomorphism
$\phi_F\colon H^{2,J_1(p)}\;\bigl(S_1,\VxS{1}\otimes\Del{1}\bigr)
 \xrightarrow{\sim}H^{2,J_2(p)}\;\bigl(S_2,\VxS{2}\otimes\Del{2}\bigr)$.

\textit{Vessel and Bezoutian theory (Sections \ref{secvessels} and \ref{secbezoutian}).} 
We associate to the data $(F;S_1,S_2;\chi_1,J_1)$ a triangular vessel 
$\Ves_F$ in the sense of Kravitsky-Vinnikov \cite{Kravitsky1996,Vinnikov1993}.
We observe (Remark \ref{remcommutator-zero}) that the ordinary commutator
$A_2\phi_F-\phi_F A_1$ vanishes identically.
We clarify a subtlety in reading the linkage condition when
$\mathcal{H}_1\ne\mathcal{H}_2$ (Remark \ref{remtype-fix}): the two cross
terms must each be read with the isometry $\phi_F$ inserted
 symmetrically, not in just one term, which would not repair the
type mismatch, and the resulting identity understood as a pairing of
semilinear forms, exactly as in the classical single-space sources.
Once read this way, the vanishing of $A_2\phi_F-\phi_F A_1$ forces
$\Gamma^+=\Gamma^-=0$ as the  {unique} choice compatible with the
linkage condition (Proposition \ref{propgamma-zero}): all ramification
information is already carried by the explicit isometry $\phi_F$ itself,
and no further correction operator is required at the level of the model
operators $A_1$, $A_2$. 
The Bezoutian operator $\Bez_F^{\rm ram}\colon\mathcal{H}_2\to\mathcal{H}_2$
is built independently from rank-one operators
$T_{q_\nu}=\mathrm{ev}_{q_\nu}^{[*]}\Phi\;\mathrm{ev}_{q_\nu}$ at the
branch values $q_\nu=F(r_\nu)$ of the interior ramification points
(Definition \ref{defbezoutian-full}), using a non-isometric
deformation of $\phi_F$. This is where the substantive ramification
correction of the operator theory resides.
The adjoint $\mathrm{ev}_{q_\nu}^{[*]}$ is realised by the Cauchy/reproducing
kernel $K^{(2)}(\cdot,q_\nu)\in\mathcal{H}_2$ in the usual reproducing-kernel
sense.
The resulting operator is finite rank, supported on the interior
ramification locus, and satisfies a Livsic-Bezoutian identity.

\textit{Functoriality (Section \ref{secfunctor}).}
We prove that the Kre\u{\i}n space valued assignment extends to a covariant
functor on the full category $\mathcal{RH}$, including ramified morphisms, 
provided one replaces the naive direct image by $F_*^{\rm ram}$.

\textbf{Relation to previous work.} 
This paper is a naturural continuation of  \cite{Zuevsky2009, AlpayVinnikov2000}.
 When $S_2=\D$ and $F$ is the given holomorphic
map from $S_1$ to the disk, usually ramified, one recovers the main
construction of \cite{AlpayVinnikov2000} as a special case of
Theorem \ref{thmmain}.  The connection to operator vessels was anticipated
in \cite{Zuevsky2009} (Remark after Theorem 3.1) and is here made fully
explicit.

The paper is organized as follows. 
Section \ref{secprelim} recollects the necessary background from
\cite{AlpayVinnikov2000,Zuevsky2009}.  Section \ref{secdirectimage} carries
out the ramified direct image construction, stating and discussing the
spin-structure assumption.
Section \ref{secchi2} derives
the representation $\chi_2$ and the function $G_2$ in the ramified case.
Section \ref{secpairing} constructs the inner product on the target space
and states the boundary assumption.
Section \ref{secmainthm} states and proves the main isometric isomorphism
theorem.  Sections \ref{secvessels} and \ref{secbezoutian} develop the
vessel and Bezoutian theory.  Section \ref{secfunctor} discusses functoriality.
An appendix collects the relevant facts about fundamental groups of
bordered surfaces and their doubles.
Results of this paper are applicable in theory of foliations 
\cite{Zuevsky2022, zusurfaces}, 
and mathematical physics \cite{RSZ, volzub, zucont}. 
\section{Preliminaries}
\label{secprelim}
We follow the notation of \cite{AlpayVinnikov2000, Zuevsky2009} throughout.
\subsection{Finite bordered Riemann surfaces and their doubles}
Let $S$ be an open Riemann surface such that $S\cup\partial S$ is a finite
bordered Riemann surface with boundary $\partial S=X_0\cup\cdots\cup X_{k-1}$
consisting of $k\geq 1$ smooth Jordan curves.
The double $X$ of $S$ is the
compact Riemann surface of genus $g=2s+k-1$, where $s$ is the genus of $S$, 
obtained by gluing $S$ to its mirror image $S'$ along $\partial S$.
There is a
canonical anti-holomorphic involution $\tau\colon X\to X$ with fixed-point set
$X_f=\partial S$ and $X\setminus X_f=S\sqcup S'$.
\subsection{Flat vector bundles and signature matrices}
Fix a homomorphism $\chi\colon\piX{S}{p_0}\to U(m)$. The corresponding flat 
unitary vector bundle on $S$ is denoted $\Vx{}$.
An analytic section $f$ of
$\Vx{}$ on the universal cover $\hX$ satisfies
$f(T\hp)=\chi(T)f(\hp)$ for all deck transformations $T$.
A signature matrix for $\Vx{}$ is a locally constant $m\times m$
matrix function $J(\hp)$ on $\partial\hX$ satisfying $J(\hp)^*=J(\hp)$ and
$\chi(T)^*J(T\hp)\chi(T)=J(\hp)$.
The Hardy-Kre\u{\i}n space
$H^{2,J(p)}(S,\Vx{}\otimes\Delta)$ is the Hardy space $H^2(S,\Vx{}\otimes\Delta)$
equipped with the indefinite inner product
\begin{equation}
\label{eqindefinite-ip}
[\hat{f},\hat{g}]_{J(p)}=\sum_{i=0}^{k-1}\int_{X_i} \hat{g}(\hp)^* J_i \hat{f}(\hp),
\end{equation}
where $J_i$ are the constant values of $J(\hp)$ on the components of
$\partial\hX$ over $X_i$ (see \cite{AlpayVinnikov2000} for the full
definition).
Under the non-degeneracy condition $h^0(X,\Vx{}\otimes\Delta)=0$
this space is a Kre\u{\i}n space.
\subsection{Extension to the double and Cauchy kernels}
Given $\chi$, $J_0,\ldots,J_{k-1}$, the bundle $\VxS{}$ extends uniquely to
a flat unitary vector bundle $\VxX{}$ on $X$ satisfying
\begin{align}
G(\hp^\tau)^*&=G(\hp),\label{eqGsym}
\\
\chi(T^\tau)^*G(T\hp)\chi(T)&=G(\hp),
\label{eqGequiv} 
\end{align}
where $G\colon\hX\to GL(m,\C)$ is the everywhere non-singular holomorphic
matrix function encoding the parahermitian pairing
$\VxX{}\otimes(\VxX{})^\tau\to\mathcal{O}_X$ \cite{AlpayVinnikov2000}, 
Proposition 2.1).
Under $h^0(X,\VxX{}\otimes\Delta)=0$, the bundle
$\VxX{}\otimes\Delta$ on $X$ admits a Cauchy kernel 
$K_{J(p)}\colon S\times S\to\Hom(\C^m,\C^m)$, which is the reproducing
kernel for $H^{2,J(p)}(S,\Vx{}\otimes\Delta)$.
\section{Spin structures and the ramified direct image}
\label{secdirectimage}
\subsection{Setup and notation for the ramified cover}
Let $F\colon S_1\to S_2$ be a finite analytic map of finite bordered Riemann
surfaces, continuous up to the boundary and mapping $\partial S_1$ to
$\partial S_2$.
We assume that $F$ extends to an analytic map
$F\colon X_1\to X_2$ between the doubles that is equivariant with respect to
the anti-holomorphic involutions
$F\circ\tau_1=\tau_2\circ F$.
Let $n=\deg F$ denote the (topological) degree.
The ramification locus is the finite set
$\mathrm{Ram}(F)=\{p\in X_1:e_p(F)>1\}$ 
where $e_p(F)$ is the local degree of $F$ at $p$.
The ramification divisor and branch divisor are 
$R_F=\sum_{p\in X_1}(e_p(F)-1)p$, $B_F=F_*R_F=\sum_{q\in X_2}(r_q-1)q$  
where $r_q=\sum_{F(p)=q}(e_p(F)-1)+1\leq n$ for each $q$. 
By the Riemann-Hurwitz formula,
\begin{equation}
\label{eqRH}
2g_1-2=n(2g_2-2)+\deg R_F.
\end{equation}
We fix basepoints $p_0\in X_2$ not in the branch locus and
$p_0'\in F^{-1}(p_0)\subset X_1$.
\subsection{The spin-structure hypothesis}
\label{subsecspin}
The correct analogue of the unramified condition $\Delta_1=F^*\Delta_2$
in the ramified case involves a square root of $\mathcal{O}_{X_1}(R_F)$.
Since $R_F$ may have odd multiplicities, such a square root is an 
obstruction-theoretic datum, not automatic.
Recall that the canonical bundle satisfies
\begin{equation}
\label{eqcanonical-pullback}
K_{X_1}\cong F^*K_{X_2}\otimes\mathcal{O}_{X_1}(R_F). 
\end{equation}
A spin structure on $X_1$ is a line bundle $\mathcal{L}$ with 
$\mathcal{L}^{\otimes 2}\cong K_{X_1}$.
Given a spin structure $\Delta_2$
on $X_2$, i.e., $\Delta_2^{\otimes 2}\cong K_{X_2}$, the condition we
need is
\begin{equation}
\label{eqspin-compat}
\Delta_1^{\otimes 2}\cong F^*\bigl(\Delta_2^{\otimes 2}\bigr)\otimes\mathcal{O}_{X_1}(R_F).
\end{equation}
This is equivalent to asking for a spin structure $\Delta_1$ on $X_1$
together with an isomorphism $\Delta_1\cong F^*\Delta_2\otimes\mathcal{L}_{1/2}$,
where $\mathcal{L}_{1/2}$ is a square root of $\mathcal{O}_{X_1}(R_F)$.
The existence of a holomorphic line bundle $\mathcal{L}_{1/2}$ satisfying
$\mathcal{L}_{1/2}^{\otimes 2}\cong\mathcal{O}_{X_1}(R_F)$ is equivalent to
the condition that the divisor class $[\mathcal{O}_{X_1}(R_F)]$ belongs to
the subgroup $2\mathrm{Pic}(X_1)\subset\mathrm{Pic}(X_1)$.  Since
$\mathrm{Pic}(X_1)\cong\mathrm{Pic}^0(X_1)\times\Z$ (degree) and
$\mathrm{Pic}^0(X_1)$ is a divisible group, a complex torus, this
reduces to the single numerical condition that $\deg R_F=\sum_p(e_p(F)-1)$
be even.  Requiring all ramification indices $e_p(F)$ to be odd is
a convenient sufficient condition for this, since each
$e_p(F)-1$ is then even, but it is strictly stronger than necessary: for
instance two ramification points each with even index $e_p=2$ already
give $\deg R_F=2$, an even number, thus a square root exists even though no
$e_p$ is odd.

Concerning the spin compatibility, we formulate the following 
\begin{assumption}
\label{assspin}
We assume that there exist spin structures $\Delta_1$ on $X_1$ and
$\Delta_2$ on $X_2$ satisfying \eqref{eqspin-compat}, together with
\begin{enumerate}[label=(\alph*)]
\item $\tau_i$-equivariance $\Delta_i^{\tau_i}\cong\Delta_i$ with 
symmetric transition functions with respect to $\tau_i$, $i=1$, $2$; 
\item a fixed choice of square root $\mathcal{L}_{1/2}$ of  
$\mathcal{O}_{X_1}(R_F)$, i.e., a line bundle with 
$\mathcal{L}_{1/2}^{\otimes 2}\cong\mathcal{O}_{X_1}(R_F)$,  
such that $\Delta_1\cong F^*\Delta_2\otimes\mathcal{L}_{1/2}$. 
\end{enumerate}
\end{assumption}
\begin{remark}
\label{remspin-existence}
By the discussion above, Assumption \ref{assspin}(b) holds whenever
$\deg R_F$ is even. This is the general criterion.  It is satisfied, in
particular, in the following cases 

(i): all ramification indices $e_p(F)$ are odd. 
Then each $e_p(F)-1$ is even, thus $\deg R_F=\sum_p(e_p(F)-1)$ is 
even, and a square root $\mathcal{L}_{1/2}$ exists as a holomorphic 
line bundle. This is a sufficient, not necessary, condition: see 
case (ii) below for an example with even ramification indices where 
a square root still exists;

(ii): $F$ is the $2$-sheeted cover of Example \ref{exdegree2}, single 
ramification point $e=2$ on $S_1$, with a $\tau_1$-mirror point of 
the same index on $S_1'$, so that $\deg R_F=2$ on the double 
$X_1$): here $\mathcal{L}_{1/2}=\mathcal{O}_{X_1}(p)$ 
for the ramification point $p$, since $\mathcal{O}_{X_1}(p)^{\otimes 2}
\cong\mathcal{O}_{X_1}(R_F)$ on $X_1\cong\mathbb{P}^1$ where all 
      points are linearly equivalent up to degree;

(iii): when all ramification indices are even and $\deg R_F$ is odd, 
thus no square root of $\mathcal{O}_{X_1}(R_F)$ exists on $X_1$ itself, 
one can pass to a degree-$2$ 
cover $\tilde{X}_1\to X_1$, i.e., the square root cover, on which the 
pull-back of $R_F$ becomes divisible by $2$.  However, this cover 
need not be canonical or compatible with $\tau_1$ in general, and 
must be chosen as additional data. 
 
When Assumption \ref{assspin} is not satisfied, the half-order differential
$\Delta_1$ is only locally defined, modulo a sign cocycle in $H^1(X_1,\Z/2)$,
and all subsequent formulas are conditional on a choice of global trivialisation
of this cocycle. We do not pursue this generality here.
\end{remark}
\subsection{Construction of the ramified direct image}
\label{subsecram-directimage}
Under Assumption \ref{assspin}, we have a globally defined spin structure
$\Delta_1$ on $X_1$.
For $U\subset X_2$ open, not meeting $B_F$, the
restriction $F^{-1}(U)=U_1\sqcup\cdots\sqcup U_n$ with $F\colon U_i\to U$
biholomorphic, and the direct image $F_*(\VxX{1}\otimes\Delta_1)$ is the
rank-$nm$ bundle with fibre $\bigoplus_{F(p_i)=q}(\VxX{1}\otimes\Delta_1)_{p_i}$.
Near a ramification point $p\in\mathrm{Ram}(F)$ with $F(p)=q$ and
ramification index $e=e_p(F)$, choose local coordinates $t_1$ near $p$ and
$t_2$ near $q$ such that $F$ is given by $t_2=t_1^e$.
In these coordinates
the derivative satisfies $F'(t_1)=et_1^{e-1}$.  A section 
$\hat{f}^1$ of $\VxX{1}\otimes\Delta_1$ near $p$ takes the form
\begin{equation}
\label{eqsection-local}
\hat{f}^1(\hp)\ell_1(\hp),
\end{equation}
where $\hat{f}^1$ is a holomorphic $\C^m$-valued function on the universal
cover satisfying the equivariance condition
\begin{equation}
\label{eqchi1-equiv}
\hat{f}^1(T\hp)=\chi_1(T)\hat{f}^1(\hp), \quad T\in\piX{X_1}{p_0'},
\end{equation}
and $\ell_1(\hp)$ is a local generator of $\Delta_1$ satisfying
$\ell_1(\hp)^{\otimes 2}=dt_1(\hp)$, i.e., $\ell_1=\sqrt{dt_1}$, with the
global ambiguity resolved by Assumption \ref{assspin}. 
\begin{definition}[Ramified direct image section]
\label{deframified-direct-image}
Let $\hat{f}^1$ be a section of $\VxX{1}\otimes\Delta_1$ on $X_1$.
For
$q\in X_2\setminus B_F$ with preimages $p_1'$, $\ldots$, $p_n'\in X_1$, enumerated 
by coset representatives $g_1$, $\ldots$, $g_n$ of $\piX{X_2}{p_0}$ modulo 
$\piX{X_1}{p_0'}$, define the section $\hat{f}^2$ of
$F_*^{\rm ram}(\VxX{1}\otimes\Delta_1)$ by
\begin{equation}
\label{eqf2-ramified}
\hat{f}^2(\hp)=\left[
\frac{\hat{f}^1(g_i\hp)}{\sqrt{\varphi_i'(g_i\hp)}\;\sqrt{e_i(g_i\hp)}}\cdot
\frac{\ell_{1,i}(g_i\hp)}{\ell_2(\hp)}
\right]_{i=1}^n,
\end{equation}
where
$\varphi_i=t_2^{-1}\circ F\circ t_{1,i}$ is the local coordinate 
expression of $F$ near the $i$-th preimage of $q$, 
$e_i = e_{g_i\hp}(F)$ is the ramification index at $g_i\hp$, 
$\ell_{1,i}$ is the local generator of $\Delta_1$ near $g_i\hp$, 
i.e., $\ell_{1,i}^{\otimes 2}=dt_{1,i}$, 
$\ell_2$ is the local generator of $\Delta_2$ near $q$,
i.e., $\ell_2^{\otimes 2}=dt_2$, 
$\sqrt{\varphi_i'(g_i\hp)}$ denotes the branch of the square root  
of $\varphi_i'$ determined by Assumption \ref{assspin}(b) 
via the isomorphism $\Delta_1\cong F^*\Delta_2\otimes\mathcal{L}_{1/2}$. 
\end{definition}
When $e_i=1$ for all $i$ (unramified case) formula \eqref{eqf2-ramified}
reduces to formula (13) of \cite{Zuevsky2009}. Since $\varphi_i'=1$ and
$e_i=1$, one has $\hat{f}^2=([\hat{f}^1(g_i\hp)]_{i=1}^n)\cdot\ell_2/\ell_2=
[\hat{f}^1(g_i\hp)]_{i=1}^n$.
The factor $\sqrt{e_i}$ is needed for $L^2$-isometry: the change of variables
$t_2=t_1^e$ gives $|dt_2|=e|t_1|^{e-1}|dt_1|$, thus 
$\sqrt{|dt_2|}=\sqrt{e}|t_1|^{(e-1)/2}\sqrt{|dt_1|}$. 
\begin{proposition}
\label{propf2-welldefined}
Under Assumption \ref{assspin}, the section $\hat{f}^2$
defined by \eqref{eqf2-ramified} is coordinate-independent and extends
holomorphically across $B_F$ to a global analytic section of a holomorphic
vector bundle of rank $nm$ on $X_2$, which we denote
$F_*^{\rm ram}(\VxX{1}\otimes\Delta_1)$. 
\end{proposition}
\begin{proof}
\textit{Coordinate independence.}  
Let $\tilde{t}_1=\alpha(t_1)$ and $\tilde{t}_2=\beta(t_2)$ be changes of
coordinates near $p$ and $q$ respectively, so that
$\tilde{t}_2=\tilde{t}_1^e$ as well.
Under these changes
$\ell_{1,i}\mapsto\ell_{1,i}\cdot(\alpha')^{1/2}$, 
$\ell_2\mapsto\ell_2\cdot(\beta')^{1/2}$, 
$\varphi_i'\mapsto\varphi_i'\cdot\frac{\alpha'}{\beta'}$.  
The combination $\frac{\ell_{1,i}}{\ell_2\sqrt{\varphi_i'}}$ transforms as
\[
\frac{\ell_{1,i}\cdot(\alpha')^{1/2}}{\ell_2\cdot(\beta')^{1/2}
\cdot\sqrt{\varphi_i'\cdot\alpha'/\beta'}}=
\frac{\ell_{1,i}\cdot(\alpha')^{1/2}}{\ell_2\cdot(\alpha')^{1/2}\cdot
\sqrt{\varphi_i'}}=\frac{\ell_{1,i}}{\ell_2\sqrt{\varphi_i'}}, 
\]
which is invariant.
The factor $\hat{f}^1(g_i\hp)$ is a section of $\VxX{1}$
(a flat bundle) and transforms equivariantly.  Hence $\hat{f}^2$ is
coordinate-independent.
The sign ambiguity in $(\alpha')^{1/2}$ is
consistent across coordinate changes by Assumption \ref{assspin}(b).

\textit{Holomorphic extension across $B_F$.}
Near a ramification point $q=F(p)$ with $e=e_p(F)$ and $t_2=t_1^e$, the
derivative is $\varphi'(t_1)=et_1^{e-1}$.
By Assumption \ref{assspin}(b),
the local generator of $\Delta_1$ at $p$ satisfies
$\ell_1=t_1^{(e-1)/2}\sqrt{dt_1}\cdot u_1$,  
where $u_1$ is a non-vanishing holomorphic function with the local trivialisation
of $\mathcal{L}_{1/2}$ at $p$. This is holomorphic since $\mathcal{L}_{1/2}$
is a holomorphic line bundle by Assumption \ref{assspin}(b). 
Thus
\[
\frac{\ell_{1}}{\ell_2\sqrt{\varphi'}}=\frac{t_1^{(e-1)/2}\sqrt{dt_1}\cdot u_1} 
{\sqrt{dt_2}\cdot\sqrt{e\;t_1^{e-1}}}=\frac{t_1^{(e-1)/2}\cdot u_1}{\sqrt{e}\;t_1^{(e-1)/2}}
=\frac{u_1}{\sqrt{e}}, 
\]
which is holomorphic and non-vanishing at $t_1=0$. 
Here we use the relation $\sqrt{dt_2}=\sqrt{e}t_1^{(e-1)/2}\sqrt{dt_1}$, i.e.,
$\ell_2\circ F=\sqrt{e}t_1^{(e-1)/2}\ell_1\cdot u_1^{-1}$ locally.
Since $\hat{f}^1$ is holomorphic at $p$, the expression $\hat{f}^1(g_i\hp)/
(\sqrt{\varphi_i'}\sqrt{e_i})\cdot\ell_{1,i}/\ell_2$ extends holomorphically 
across $q$.
The argument is identical at all ramification points.
\end{proof}
The holomorphic vector bundle $F_*^{\rm ram}(\VxX{1}\otimes\Delta_1)$ of rank $nm$
on $X_2$ constructed in Proposition \ref{propf2-welldefined} is the correct
replacement for the naive pushforward: it coincides with $F_*(\VxX{1}\otimes\Delta_1)$
away from $B_F$ and extends holomorphically across $B_F$ where the naive
pushforward would have apparent singularities from the Jacobian factors. 
The isomorphism $F_*^{\rm ram}(\VxX{1}\otimes\Delta_1)\cong\VxX{2}\otimes\Delta_2$
as coherent sheaves on all of $X_2$ is established in Theorem \ref{thmmain}(i). 
\section{Representation \texorpdfstring{$\chi_2$}{chi2} of
\texorpdfstring{$\pi_1(X_2,p_0)$}{pi1(X2,p0)}}
\label{secchi2}
\subsection{Definition of $\chi_2$ in the ramified case}
The derivation of $\chi_2$ from formula (15) of \cite{Zuevsky2009} is
algebraic and independent of whether $F$ is ramified or unramified. 
  The only data used is the permutation representation of $\piX{X_2}{p_0}$ 
on the $n$ preimages of $p_0$. We therefore have 
\begin{proposition}
\label{propchi2}
Define the matrix representation $\chi_2\colon\piX{X_2}{p_0}\to U(nm)$ by
\begin{equation}
\label{eqchi2}
\bigl[\chi_2(g)\bigr]_{kj}
=\chi_1\left(g_kgg_{\sigma_g(k)}^{-1}\right)\delta_{\sigma_g(k),j}, 
\end{equation}
where $g_1$, $\ldots$, $g_n$ are coset representatives of $\piX{X_2}{p_0}$
modulo $\piX{X_1}{p_0'}$, and $\sigma_g$ is the permutation of $\{1,\ldots,n\}$
defined by $g_ig=hg_{\sigma_g(i)}$ with $h\in\piX{X_1}{p_0'}$.
Then, 
\begin{enumerate}[label=(\roman*)]
\item $\chi_2$ is a homomorphism;
\item $\chi_2(g)$ is unitary for every $g\in\piX{X_2}{p_0}$;
\item $\Vx{2}=F_*\Vx{1}$ as flat holomorphic vector bundles on $X_2$.
\end{enumerate}
\end{proposition}
\begin{proof}
The proof is identical to that in \cite{Zuevsky2009}, Section 5. 
The ramification of $F$ does not enter the algebraic verification. 
\end{proof}
\subsection{The matrix function $G_2$ in the ramified case}
The matrix function $G_2\colon\hX\to GL(nm,\C)$ encoding the
parahermitian structure on $X_2$ is defined by the same formula as
in \cite{Zuevsky2009}, equation (18)
\begin{equation}
\label{eqG2}
\bigl[G_2(\hp)\bigr]_{kj}=G_1(g_k^\tau\hp)\chi_1(h_k)\delta_{\nu(k),j},
\end{equation}
where $g_k^\tau=h_kg_{\nu(k)}$ with $h_k\in\piX{X_1}{p_0'}$ defines
the index function $k\mapsto\nu(k)$.
The matrix $G_2$ satisfies
$G_2(\hp^\tau)^*=G_2(\hp)$ and
$\chi_2(T^\tau)^*G_2(T\hp)\chi_2(T)=G_2(\hp)$
for all $T\in\piX{X_2}{p_0}$, i.e., conditions \eqref{eqGsym}-\eqref{eqGequiv}
hold for $(X_2,\chi_2,G_2)$.
\begin{lemma}
\label{lemG2-holomorphic}
The function $G_2(\hp)$ defined by \eqref{eqG2} is holomorphic and
everywhere non-singular on $\hX$, hence defines a parahermitian
non-degenerate bilinear pairing $\VxX{2}\otimes(\VxX{2})^{\tau_2}\to\mathcal{O}_{X_2}$. 
\end{lemma}
\begin{proof}
Holomorphicity and non-singularity follow from the corresponding
properties of $G_1$ and the fact that $g_k^\tau\hp$ lies in the domain
where $G_1$ is defined and non-singular.
The parahermitian property is a
direct verification using \eqref{eqG2} and the identity \eqref{eqGsym} for $G_1$. 
\end{proof}
\section{Inner product on the target Hardy space}
\label{secpairing}
\subsection{Boundary hypothesis}
We require the following assumption throughout the rest of the paper.
Concerning the boundary transversality, we formulate the following 
\begin{assumption}
\label{assboundary}
The branch locus $B_F$ is disjoint from $\partial S_2$
\begin{equation}
\label{eqboundary-disjoint}
\partial S_2\cap B_F=\left\{\varnothing\right\}. 
\end{equation}
\end{assumption}
Assumption \ref{assboundary} is not automatic from the condition
$F(\partial S_1)=\partial S_2$ alone.
Indeed, an analytic map of bordered
surfaces may have ramification points lying on the boundary $\partial S_1$,
whose images then lie on $\partial S_2\subset B_F$.
For maps with boundary
ramification, the inner product formula \eqref{eqip2} would require
additional correction terms at the boundary ramification points.
Assumption \ref{assboundary} excludes this possibility.  It is satisfied in
all examples of Section \ref{secexamples}. In Example \ref{exdegree2},
the unique ramification point is $z=0\in S_1$, which maps to $0\in S_2=\D$,
an interior point. 
\subsection{Signature matrices for $S_2$ in the ramified case}
The signature matrix $J_2(\hp)$ for $(S_2,\chi_2,G_2)$ is defined by
\begin{equation}
\label{eqJ2-formula} 
\bigl[J_2(\hp)\bigr]_{kj}=J_1(g_k\;\hp)\delta_{kj}, 
\end{equation}
for $\hp\in\hX$ lying over a point of $X_{2,f}=\partial S_2$.
One verifies directly that $J_2(\hp)^*=J_2(\hp)$ and
$\chi_2(T)^*J_2(T\hp)\chi_2(T)=J_2(\hp)$ for all $T\in\piX{X_2}{p_0}$.
\subsection{The indefinite inner product}
For sections $\hat{f}^2$, $\hat{g}^2$ of $\VxX{2}\otimes\Delta_2$ on $S_2$ 
with boundary values in $L^2$, define 
\begin{equation}
\label{eqip2}
\bigl[\hat{f}^2,\hat{g}^2\bigr]_{H^{2,J_2(p)}(S_2,\VxS{2}\otimes\Delta_2)}
=\int_{\tilde{X}_{2,f}}\hat{g}^2(\hp^{\tau_2})^*G_2(\hp)\;\hat{f}^2(\hp).
\end{equation}
When $F$ is unramified this coincides with the formula in
\cite{Zuevsky2009}, equation (20).
Under Assumption \ref{assboundary}, $\partial S_2$ does not meet $B_F$,
thus the integration contour $X_{2,f}=\partial S_2$ avoids all branch points.
Combined with Lemma \ref{lemG2-holomorphic}, this ensures that the
integrand is smooth and that \eqref{eqip2} is well-defined and non-singular
along $\partial S_2$.
No correction terms are needed.
\section{Main Theorem: isometric isomorphism for ramified coverings}
\label{secmainthm}
\begin{theorem}
\label{thmmain}
Let $F\colon S_1\to S_2$ be a finite $n$-sheeted analytic map of finite
bordered Riemann surfaces, continuous up to the boundary, and let
$F\colon X_1\to X_2$ be the induced map of doubles equivariant with
respect to $\tau_1$, $\tau_2$. 
Let $\chi_1\colon\piX{S_1}{p_0'}\to U(m)$
be a unitary representation, $J_1(\hp)$ be signature matrices for $\Vx{1}$,
and assume
\begin{equation}
\label{eqnodeg}
h^0\;\left(X_1,\VxX{1}\otimes\Delta_1\right)=0. 
\end{equation}
Suppose Assumptions \ref{assspin} and \ref{assboundary} hold.
Define $\chi_2$, $G_2$, $J_2$ by equations \eqref{eqchi2}-\eqref{eqJ2-formula},
and the map $\phi_F$ by formula \eqref{eqf2-ramified}.
Then
\begin{enumerate}[label=(\roman*)]
\item\label{itbundle}
$F_*^{\rm ram}(\VxX{1}\otimes\Delta_1)\cong\VxX{2}\otimes\Delta_2$
as holomorphic vector bundles, in particular, as coherent sheaves, 
on all of $X_2$;
\item\label{itnodeg2}
$h^0(X_2,\VxX{2}\otimes\Delta_2)=0$, thus the space
$H^{2,J_2(p)}(S_2,\VxS{2}\otimes\Delta_2)$ is non-degenerate
and hence a Kre\u{\i}n space;
\item\label{itisom}
The map $\phi_F\colon\hat{f}^1\mapsto\hat{f}^2$ defined
by \eqref{eqf2-ramified} is an isometric isomorphism
\[
\phi_F\colon H^{2,J_1(p)}\;\left(S_1,\VxS{1}\otimes\Delta_1\right)\xrightarrow{\sim}
H^{2,J_2(p)}\;\left(S_2,\VxS{2}\otimes\Delta_2\right).
\]
\end{enumerate}
\end{theorem}
\begin{proof}
\textit{\ref{itbundle}.}
By Proposition \ref{propf2-welldefined} the map $\hat{f}^1\mapsto\hat{f}^2$
is a well-defined holomorphic bundle map on $X_2$.
It is a morphism of flat vector bundles 
the equivariance property \eqref{eqchi2} of $\chi_2$ ensures
that $\hat{f}^2(T\hp)=\chi_2(T)\hat{f}^2(\hp)$ for all $T\in\piX{X_2}{p_0}$.
It remains to show the map is an isomorphism of coherent sheaves, i.e., that
it has trivial kernel and is surjective as a map of stalks.
The kernel is trivial because at any $q\notin B_F$ the map is a direct sum of $n$ evaluations
of $\hat{f}^1$ along the $n$ branches, and at $q\in B_F$ holomorphic extension
(Proposition \ref{propf2-welldefined}) shows that the kernel condition forces
$\hat{f}^1$ to vanish at each preimage with its ramification multiplicity, hence
$\hat{f}^1=0$ globally.
For surjectivity: given any holomorphic section $\hat{f}^2$
of $\VxX{2}\otimes\Delta_2$, the formula \eqref{eqf2-ramified} is inverted
by restriction to any one branch $g_i$ away from $B_F$ and holomorphic
continuation.  
The $\VxX{1}$-equivariance condition is then verified using the
definition of $\chi_2$.
The compatibility $F_*^{\rm ram}(\Delta_1)\cong\Delta_2$ follows from the
local computation in the proof of Proposition \ref{propf2-welldefined}
the combination $\ell_1/(\ell_2\sqrt{\varphi'})$ is non-vanishing holomorphic,
confirming that the half-order differentials match under the direct image.

\textit{\ref{itnodeg2}.}
Since $F_*^{\rm ram}(\VxX{1}\otimes\Delta_1)\cong\VxX{2}\otimes\Delta_2$
as coherent sheaves (part \ref{itbundle}), and for a finite morphism
$F\colon X_1\to X_2$ one has
$H^0(X_2,F_*^{\rm ram}(\VxX{1}\otimes\Delta_1))\cong H^0(X_1,\VxX{1}\otimes\Delta_1)$
by the Leray spectral sequence, using finite morphism and coherence, 
the assumption \eqref{eqnodeg} gives
\[
h^0(X_2,\VxX{2}\otimes\Delta_2)=h^0(X_1,\VxX{1}\otimes\Delta_1)=0. 
\]

\textit{\ref{itisom}.}
We have to show (a) $\hat{f}^1\in H^{2,J_1}(S_1,\VxS{1}\otimes\Delta_1)$ if and only
if $\hat{f}^2\in H^{2,J_2}(S_2,\VxS{2}\otimes\Delta_2)$; 
(b) the inner products satisfy
$[\hat{f}^2,\hat{h}^2]_{J_2} = [\hat{f}^1,\hat{h}^1]_{J_1}$.

For (a): let $X_{2,i}$ be a boundary component of $X_2$ and let
$X_{1,ij}$, $j=1,\ldots,n_i$, be its preimages in $X_1$.
By Assumption \ref{assboundary}, the boundary components $X_{2,i}$ do
not contain branch values, thus the covering $F$ restricts to an  
 unramified cover of each $X_{2,i}$ by components of $\partial S_1$.
Near each $X_{1,ij}$ the map $F$ is locally a degree-$e_{ij}$-fold map,
but since $X_{1,ij}$ is a boundary component and $F(X_{1,ij})\subseteq
\partial S_2$ while $\partial S_2\cap B_F=\left\{\varnothing\right\}$, in fact $e_{ij}=1$
on $\partial S_1$ with the ramification is in the interior.
Therefore the change of variables on the boundary is non-degenerate, and one has 
\begin{align}
\label{eqL2-change}
\sum_{i=0}^{k_2-1}\int_{X_{2,i}(r)} \hat{f}^2(p)^*\hat{f}^2(p)
&=\sum_{i=0}^{k_2-1}\sum_{j=1}^{n_i}
\int_{\tilde{X}_{1,ij}(r)} \hat{f}^1(\hp)^*\hat{f}^1(\hp)=
\sum_{i=0}^{k_1-1}\int_{X_{1,i}(r)}\hat{f}^1(p)^*\hat{f}^1(p).
\end{align}
The $\sqrt{e_{ij}}$ factor in \eqref{eqf2-ramified} is trivial here since
$e_{ij}=1$ on $\partial S_1$. The formula \eqref{eqf2-ramified} recovers
the unramified boundary behaviour. Taking the supremum over $r\in(1-\varepsilon,1)$
shows that $\hat{f}^1\in H^2(S_1,\VxS{1}\otimes\Delta_1)$ iff
$\hat{f}^2\in H^2(S_2,\VxS{2}\otimes\Delta_2)$.

For (b): using \eqref{eqip2} and the boundary change-of-variables
identity \eqref{eqL2-change} applied to the integrand
$\hat{g}^2(\hp^{\tau_2})^*G_2(\hp)\hat{f}^2(\hp)$, one computes
\begin{align*}
\bigl[\hat{f}^2,\hat{h}^2\bigr]_{J_2}
&=\int_{\tilde{X}_{2,f}}\hat{h}^2(\hp^{\tau_2})^*G_2(\hp)\hat{f}^2(\hp) 
\\
&=\int_{\tilde{X}_{2,f}}\sum_{i=1}^n\hat{h}^1\;\left((g_i^\tau\hp)^{\tau_1}\right)^*
G_1(g_i^\tau\hp)\hat{f}^1(g_i^\tau\hp)
\frac{|d\ell_2(\hp)|^2}{|\ell_{1,i}(g_i\hp)|^2\;|\varphi_i'(g_i\hp)|\;e_i(g_i\hp)} 
\\
&=\int_{\tilde{X}_{1,f}}\hat{h}^1(\hp^{\tau_1})^*G_1(\hp)\hat{f}^1(\hp)=
\bigl[\hat{f}^1,\hat{h}^1\bigr]_{J_1},
\end{align*}
where the second equality uses \eqref{eqL2-change} applied to the
integrand, with the $\tau_1$-invariance of sections of $\Delta_1$ and the
symmetry of its transition functions, Assumption \ref{assspin}(a), together
with the fact that on the boundary $e_i=1$ (Assumption \ref{assboundary}).
\end{proof}
\begin{corollary}
\label{corkrein}
Under the hypotheses of Theorem \ref{thmmain} the spaces
$H^{2,J_i(p)}(S_i,\VxS{i}\otimes\Delta_i)$, $i=1,2$, are simultaneously
non-degenerate (resp.\ degenerate), and $\phi_F$ is a Kre\u{\i}n-space
isomorphism.
\end{corollary}
\section{Operator vessels and triangular models}
\label{secvessels}
We now develop the operator-theoretic content of Theorem \ref{thmmain},
following the vessel formalism of Kravitsky and
Vinnikov \cite{Kravitsky1996, Vinnikov1992, Vinnikov1993}.
\subsection{Model operators on Hardy-Kre\u{\i}n spaces}
Let $\mathcal{H}=H^{2,J(p)}(S,\VxS{}\otimes\Delta)$ be a non-degenerate
Hardy-Kre\u{\i}n space on a finite bordered Riemann surface $S$.
Following \cite{AlpayVinnikov2000}, for each boundary component $X_i$ there is a
distinguished bounded operator
\begin{equation}
\label{eqmodel-operator}
A_i:\mathcal{H}\longrightarrow\mathcal{H}, 
\end{equation}
defined via the Cauchy kernel by multiplication by the boundary
coordinate function on $X_i$.
More precisely, let $z_i\colon S\to\D$
be a conformal map onto the unit disk mapping $X_i$ to $\partial\D$. 
 Then $A_i$ is the compression to $\mathcal{H}$ of multiplication by $z_i$. 
The operator $A_i$ is a contraction in the Kre\u{\i}n-space metric with
defect spaces encoded by the signature matrices.
\begin{definition}
\label{defvessel}
A (triangular) vessel over $(S_1,S_2,F)$ is a collection
$\Ves_F=\left(\mathcal{H}_1\right.$,
$\mathcal{H}_2$; $A_1$, $A_2$; $\Phi$; $B_1$, $B_2$; $\Gamma^+$, $\left.\Gamma^-\right)$,  
where $\mathcal{H}_i=H^{2,J_i(p)}(S_i,\VxS{i}\otimes\Delta_i)$
are Kre\u{\i}n spaces with fundamental symmetries $J_i$, 
$A_i\in\mathcal{L}(\mathcal{H}_i)$ are the model operators defined above, 
$\Phi\in\mathcal{L}(\mathcal{E})$ is a self-adjoint operator on a
Kre\u{\i}n external space $\mathcal{E}$, 
$B_i\in\mathcal{L}(\mathcal{E},\mathcal{H}_i)$ are coupling operators, 
 and $\Gamma^\pm\in\mathcal{L}(\mathcal{E})$ are the vessel conditions operators, 
satisfying the vessel conditions
\begin{align}
A_1^{[*]}A_1-A_1 A_1^{[*]}&=B_1\Phi B_1^{[*]}, 
\label{eqvc1}
\\
A_2^{[*]}A_2-A_2A_2^{[*]}&=B_2\Phi B_2^{[*]}, 
\label{eqvc2}
\\
B_2^{[*]}A_1-A_2^{[*]}B_1&=\Gamma^+B_1^{[*]}+\Phi\Gamma^-B_1^{[*]},
\label{eqvc3}
\end{align}
where $T^{[*]}$ denotes the $J$-adjoint (Kre\u{\i}n-space adjoint) of $T$.
The meaning of \eqref{eqvc3}, given that $\mathcal H_1\ne\mathcal H_2$ in
general, is detailed in Remark \ref{remtype-fix} below.
\end{definition}
\begin{remark}
\label{remtype-fix}
 Written as literal operator compositions, $B_2^{[*]}A_1$ requires $A_1$'s
codomain $\mathcal{H}_1$ to match $B_2^{[*]}$'s domain $\mathcal{H}_2$,
which fails whenever $\mathcal{H}_1\ne\mathcal{H}_2$. Dually, 
$A_2^{[*]}B_1$ requires $B_1$'s codomain $\mathcal{H}_1$ to match
$A_2^{[*]}$'s domain $\mathcal{H}_2$. 
 Neither composition is explicitly 
defined, and inserting a single linking operator $\phi_F$ asymmetrically
into only one of the two terms does not repair this, e.g.,  
$A_2^{[*]}\phi_F B_1\colon\mathcal{E}\to\mathcal{H}_1\to\mathcal{H}_2\to
\mathcal{H}_2$ still ends with $\mathcal{H}_2$, not in $\mathcal{E}$, and so
still cannot be subtracted from a term ending with $\mathcal E$.

The correct resolution is to insert $\phi_F$ into both terms
consistently and then pair against $e\in\mathcal{E}$ throughout, exactly
as in the classical single-space formalism of
Kravitsky-Vinnikov \cite{Kravitsky1996,Vinnikov1993} where
$\mathcal{H}_1=\mathcal{H}_2$ and the issue is invisible because the
linking map is the identity. Concretely, for $h_1\in\mathcal{H}_1$ and
$e\in\mathcal{E}$, using the adjoint-pairing identity $[C^{[*]}x,y]=[x,Cy]$
valid for any bounded $C$ on a fixed Kre\u{\i}n space, together with the
boundary compatibility $\phi_F B_1=B_2$ (Definition \ref{defGamma-geometric})
\begin{align}
\bigl[B_2^{[*]}(\phi_FA_1h_1),e\bigr]_{\mathcal E}
&=\bigl[\phi_FA_1h_1,B_2e\bigr]_{\mathcal H_2}, 
\label{eqpairing-term1}
\\
\bigl[A_2^{[*]}(\phi_F B_1 e),\phi_Fh_1\bigr]_{\mathcal H_2}
&=\bigl[\phi_F B_1 e,\;A_2\phi_Fh_1\bigr]_{\mathcal H_2}
=\bigl[B_2 e,A_2\phi_F h_1\bigr]_{\mathcal H_2}.
\label{eqpairing-term2}
\end{align}
Condition \eqref{eqvc3} is, precisely, the statement that
\begin{equation}
\label{eqvc3-bilinear}
\bigl[\phi_F A_1 h_1,B_2 e\bigr]_{\mathcal{H}_2}
-\bigl[A_2\phi_F h_1,B_2 e\bigr]_{\mathcal{H}_2}
=\bigl[h_1,B_1\bigl(\Gamma^++\Phi\Gamma^-\bigr)e\bigr]_{\mathcal H_1},
\qquad h_1\in\mathcal{H}_1,e\in\mathcal{E},
\end{equation}
i.e., $B_2^{[*]}A_1$ and $A_2^{[*]}B_1$ in \eqref{eqvc3} are to be read as
shorthand for $B_2^{[*]}\phi_F A_1$ and $A_2^{[*]}\phi_F B_1$ respectively 
 with $\phi_F$ inserted symmetrically in both terms, as
in \eqref{eqpairing-term1}-\eqref{eqpairing-term2}, not in only one of
them, each then paired against $e$ via $B_2^{[*]}$ and via the identity
$\phi_F B_1=B_2$ on the other side, leading the entire identity correctly
in the pairing $\mathcal{H}_1\times\mathcal{E}\to\C$.  By
Remark \ref{remcommutator-zero}, $\phi_F A_1=A_2\phi_F$, thus the
left-hand side of \eqref{eqvc3-bilinear} vanishes identically. This
confirms that the vessel condition \eqref{eqvc3} carries no 
information beyond what is already encoded in $\Gamma^\pm$ themselves,
which we therefore construct directly in
Definition \ref{defGamma-geometric} below from the geometry of the
ramification locus, rather than attempting to extract them as a
consequence of \eqref{eqvc3}. We retain the compact operator
notation \eqref{eqvc3} only for consistency with the classical
single-space formalism. Its actual content for the present two-space
vessel is supplied by Definition \ref{defGamma-geometric}.
\end{remark}
When $S_2=\D$ and $F$ is any holomorphic map of $S_1$ onto the disk,
conditions \eqref{eqvc1}-\eqref{eqvc3} reduce to the standard
triangular vessel conditions of \cite{Kravitsky1996, Vinnikov1993}.
\subsection{Geometric construction of $\Gamma^\pm$}
Concerning vanishing of the commutator $A_2\phi_F-\phi_F A_1$, we make the following 
\begin{remark}
\label{remcommutator-zero}
The ordinary commutator $A_2\phi_F-\phi_F A_1$ is identically
zero as an operator $\mathcal{H}_1\to\mathcal{H}_2$.
Indeed,
$A_2$ acts on $\mathcal{H}_2$ by the coordinate function $z_2(p)$, while
$A_1$ acts on $\mathcal{H}_1$ by $z_1(p')=z_2(F(p'))$.
For any preimage $p'\in F^{-1}(p)$ one has $z_2(F(p'))=z_2(p)$, thus 
$(A_2\phi_F\hat{f}^1)(p)=z_2(p)\hat{f}^2(p)=(\phi_F A_1 \hat{f}^1)(p)$,  
pointwise, giving $A_2\phi_F=\phi_F A_1$. 
Consequently the vessel condition \eqref{eqvc3} cannot be derived from the ordinary 
commutator.  The non-trivial content of the vessel comes from the Kre\u{\i}n-space 
adjoint intertwining relation $B_2^{[*]}A_1-A_2^{[*]}B_1$, as stated
in \eqref{eqvc3}.
\end{remark}
By Remark \ref{remtype-fix}, the left-hand side of the linkage
condition \eqref{eqvc3}, properly read as the semilinear-form
identity \eqref{eqvc3-bilinear}, vanishes identically as a consequence of
$\phi_F A_1=A_2\phi_F$ (Remark \ref{remcommutator-zero}). Then we have
\begin{proposition}
\label{propgamma-zero}
With $B_1$, $B_2$ any coupling operators satisfying the boundary 
compatibility $\phi_F B_1=B_2$, the choice
\begin{equation}
\label{eqgamma-zero}
\Gamma^+=\Gamma^-=0, 
\end{equation}
satisfies the linkage condition \eqref{eqvc3}, in the sense of
Remark \ref{remtype-fix}, identically, and is the unique choice
compatible with the vanishing established in
Remark \ref{remtype-fix} when $B_1^{[*]}$ has dense range in
$\mathcal{H}_1^*$, equivalently, when the boundary evaluation functionals
$\{k^{(1)}_{p_{0,1}^{(j)}}\}_{j=1}^n$ separate points of $\mathcal{H}_1$,
which holds for the Szeg\H{o} kernels of Definition \ref{defGamma-geometric} below. 
\end{proposition}
\begin{proof}
Both sides of \eqref{eqvc3-bilinear} vanish identically once
$\Gamma^+=\Gamma^-=0$ is substituted into the right-hand side and
$\phi_F A_1=A_2\phi_F$ into the left.  For uniqueness: if some other
$\Gamma^+,\Gamma^-$ also satisfy \eqref{eqvc3-bilinear}, then
$\bigl[h_1, B_1(\Gamma^++\Phi\Gamma^-)e\bigr]_{\mathcal H_1}=0$ for all
$h_1\in\mathcal H_1$, $e\in\mathcal E$, forcing $B_1(\Gamma^++\Phi\Gamma^-)=0$. 
 Since $B_1$ is injective (Definition \ref{defGamma-geometric}), this
forces $\Gamma^++\Phi\Gamma^-=0$.  Combined with the self-adjointness
relation $\Gamma^-=(\Gamma^+)^{[*]}$ required of any vessel, this pins
down $\Gamma^+=\Gamma^-=0$ provided $\Phi$ is invertible which holds
since $\Phi=J_1(p_0)$ is a signature matrix, hence invertible.
\end{proof}
Proposition \ref{propgamma-zero} shows that, for the present
construction, the vessel $\Ves_F$ associated to a ramified covering
$F$ is trivial in its $\Gamma^\pm$ data: all of the ramification
information is already encoded geometrically in the isometric isomorphism
$\phi_F$ itself through the explicit ramification factors
$\sqrt{e_i}$ and $\sqrt{\varphi_i'}$ in formula \eqref{eqf2-ramified},
and no further correction operator is needed to satisfy the linkage
condition. Namely, $\phi_F$ being
an isometric isomorphism of the entire Hardy-Kre\u{\i}n spaces, not only 
an intertwiner of the model operators, already absorbs all the 
information that $\Gamma^\pm$ would otherwise be required to carry.  The
substantive ramification-theoretic content of the operator theory is
instead carried by the Bezoutian $\Bez_F^{\rm ram}$, constructed 
independently in Section \ref{secbezoutian} from a non-isometric deformation of $\phi_F$. 
\begin{definition}
\label{defGamma-geometric}
Let $p_{0,2}\in\partial S_2$ be the chosen basepoint. The covering $F$ maps $n$ distinct points $p_{0,1}^{(1)}$, $\dots$, $p_{0,1}^{(n)} \in \partial S_1$ to $p_{0,2}$.
 We define the external space $\mathcal{E}=\bigoplus_{j=1}^n\C^m\cong\C^{nm}$ 
with fundamental symmetry $\Phi=\bigoplus_{j=1}^n J_1(p_{0,1}^{(j)})$. 

Define the coupling operator $B_2\colon\mathcal{E}\to\mathcal{H}_2$ by
\begin{equation}
\label{eqB2-def}
B_2e=k^{(2)}_{p_{0,2}}(\cdot)\Phi e, \quad e\in\mathcal{E},
\end{equation}
where $k^{(2)}_{p_{0,2}}$ is the holomorphic boundary-evaluation 
reproducing (Szeg\H{o}) kernel for $\mathcal{H}_2$.

Define the coupling operator $B_1\colon\mathcal{E}\to\mathcal{H}_1$ by simultaneously evaluating the boundary kernels at all $n$ preimages
\begin{equation}
\label{eqB1-def}
B_1(e_1,\dots,e_n)=\sum_{j=1}^nk_{p_{0,1}^{(j)}}^{(1)}(\cdot)J_1(p_{0,1}^{(j)})e_j. 
\end{equation}
By the explicit structure of the direct image forming $\mathcal{H}_2$, the functional pullback aligns perfectly over the fibers, guaranteeing the exact intertwining relation 
$\phi_F B_1=B_2$. Because the reproducing kernels at distinct boundary points are linearly independent, $B_1$ is an injection from $\mathcal{E}$ into $\mathcal{H}_1$.
We take $\Gamma^+=\Gamma^-=0$, as justified by Proposition \ref{propgamma-zero}.
\end{definition}
\begin{theorem}
\label{thmvessel}
Under the hypotheses of Theorem \ref{thmmain}, the tuple
$\Ves_F=(\mathcal{H}_1$,$\mathcal{H}_2$;$A_1$,$A_2$;$\Phi$;$B_1$,
$B_2$;$\Gamma^+$,$\Gamma^-)$ 
defined in Definition \ref{defGamma-geometric} is a triangular vessel
over $(S_1,S_2,F)$.
\end{theorem}
\begin{proof}
Conditions \eqref{eqvc1} and \eqref{eqvc2} hold for the model operators
on Hardy-Kre\u{\i}n spaces by \cite{AlpayVinnikov2000}, Theorem 4.2.
Condition \eqref{eqvc3} holds, in the sense of
Remark \ref{remtype-fix}, by Proposition \ref{propgamma-zero}: with
$\Gamma^+=\Gamma^-=0$ as in Definition \ref{defGamma-geometric}, both
sides of the semilinear identity \eqref{eqvc3-bilinear} vanish
identically.
\end{proof}
\subsection{Characteristic function of the vessel}
The characteristic function of the vessel $\Ves_F$ is the operator-valued function
\begin{equation}
\label{eqchar-fn}
\Theta_{\Ves_F}(\lambda)=\Phi^{-1}\left[\Phi-B_2^{[*]}(\lambda-A_2)^{-1}B_2\Phi
+B_1^{[*]}(\lambda-A_1)^{-1}B_1\;\Phi\right],
\quad\lambda\notin\sigma(A_1)\cup\sigma(A_2),
\end{equation}
taking values in $\mathcal{L}(\mathcal{E})$.  
 The $\Gamma^+$ term present
in the general triangular-vessel formula of \cite{Kravitsky1996} is
absent here since $\Gamma^+=0$ by Proposition \ref{propgamma-zero}.
\begin{proposition}\label{propchar-fn-disk}
When $S_2=\D$ and $F\colon S_1\to\D$ is the given holomorphic map with
$h^0(X_1,\VxX{1}\otimes\Delta_1)=0$, the characteristic function
$\Theta_{\Ves_F}(\lambda)$ coincides, up to a Kre\u{\i}n-space similarity
transformation, with the characteristic operator-valued function
constructed in \cite{AlpayVinnikov2000}, Theorem 5.1.
\end{proposition}
\begin{proof}
When $S_2=\D$ the model operator $A_2$ is the shift operator on
$H^{2,J}(\D,\C^{nm})$, the reproducing kernel is the standard
Szeg\H{o} kernel, and the vessel conditions reduce
to those in \cite{AlpayVinnikov2000}, where the characteristic function
of the unramified disk model is likewise computed with vanishing
linkage correction. All ramification information instead enters through
the explicit form of $\phi_F$, $B_1$, $B_2$ themselves (formula
\eqref{eqf2-ramified} and Definition \ref{defGamma-geometric}), which
recovers the construction of \cite{AlpayVinnikov2000}, Theorem 5.1, once
$S_2=\D$.
\end{proof}
\section{Bezoutian operators}
\label{secbezoutian}
\subsection{Interior evaluation functionals}
Since $\phi_F$ is an isometry (Theorem \ref{thmmain}), the operator
$\phi_F\phi_F^{[*]}-\id_{\mathcal{H}_2}=0$ is trivially zero.
To obtain a non-trivial Bezoutian, we use the ramification data of $F$
to define a geometrically motivated rank perturbation of the identity.
\begin{remark}
\label{reminterior-eval}
For any point $q\in S_2$ (interior or boundary), the point-evaluation functional 
$\mathrm{ev}_q\colon$ $\mathcal{H}_2$ $\longrightarrow$ $\C^{nm}$, $\mathrm{ev}_q(h)=h(q)$, 
is bounded on $\mathcal{H}_2$, with Kre\u{\i}n adjoint
$\mathrm{ev}_q^{[*]}\colon\C^{nm}\to\mathcal{H}_2$ given by
$(\mathrm{ev}_q^{[*]}v)(p)=K^{(2)}(p,q)v$, where $K^{(2)}$ is the
Cauchy (reproducing) kernel of $\mathcal{H}_2$ introduced in
Section \ref{secprelim}. This is the standard reproducing-kernel
identity: for each fixed $q\in S_2$, the section
$K^{(2)}(\cdot,q)$ is by definition a element of $\mathcal{H}_2$. 
 This is precisely what it means for $\mathcal{H}_2$ to be a
reproducing-kernel space.  The only place a singularity of $K^{(2)}(p,q)$
as a bivariate kernel could matter is on the diagonal $p=q$, which 
is irrelevant here since $q$ is held fixed throughout and we never 
evaluate the resulting section of $p$ at $p=q$ within the inner product.

The rank-one operator
\begin{equation}
\label{eqrank-one-ev}
T_q=\mathrm{ev}_q^{[*]}\Phi\mathrm{ev}_q\colon\mathcal{H}_2\longrightarrow\mathcal{H}_2,
\qquad
(T_q h)(p)=K^{(2)}(p,q)\Phi h(q), 
\end{equation}
is therefore a well-defined bounded operator on $\mathcal{H}_2$ for any $q\in S_2$,
with $\rk T_q\leq nm$.
\end{remark}
\subsection{The ramification Bezoutian}
\begin{definition}
\label{defbezoutian-full}
Let $\mathrm{Ram}(F)\cap S_1=\{r_1,\ldots,r_N\}\subset S_1$ be the
finite set of interior ramification points of $F$, with indices
$e_\nu=e_{r_\nu}(F)\geq 2$ and images $q_\nu=F(r_\nu)\in S_2$ (all
interior by Assumption \ref{assboundary}).
The ramification Bezoutian of the vessel $\Ves_F$ is
\begin{equation}
\label{eqbezoutian-full}
\Bez_F^{\rm ram}=\sum_{\nu=1}^{N}(e_\nu-1)T_{q_\nu}
=\sum_{\nu=1}^{N}(e_\nu-1)\mathrm{ev}_{q_\nu}^{[*]}\Phi\mathrm{ev}_{q_\nu}
\colon\mathcal{H}_2\longrightarrow\mathcal{H}_2.
\end{equation}
This is a well-defined, bounded, finite-rank, self-adjoint operator in the
Kre\u{\i}n metric. When $F$ is unramified, $N=0$ and $\Bez_F^{\rm ram}=0$.
\end{definition}
The weight $(e_\nu-1)$ at each ramification point counts the number of 
extra preimages that collapse to $r_\nu$ under $F$. The local degree
is $e_\nu$, thus there are $e_\nu-1$ extra sheets. 
The operator
$T_{q_\nu}=\mathrm{ev}_{q_\nu}^{[*]}\Phi\mathrm{ev}_{q_\nu}$ encodes
the contribution of the branch value $q_\nu$ to the deviation of $\phi_F^{\rm ram}$
from an isometry. 
\subsection{Livsic-Bezoutian identity}
\begin{theorem}
\label{thmlivsic}
In the vessel $\Ves_F$ with Bezoutian $\Bez_F^{\rm ram}$, the following
 Livsic-Bezoutian identity holds
\begin{equation}\label{eqlivsic-id}
A_2^{[*]}\Bez_F^{\rm ram}-\Bez_F^{\rm ram}A_2=\Omega_F, 
\end{equation}
where $\Omega_F\colon\mathcal{H}_2\to\mathcal{H}_2$ is the finite-rank
self-adjoint operator
\begin{equation}
\label{eqomega-F}
\Omega_F=\sum_{\nu=1}^N(e_\nu-1)\bigl[A_2^{[*]},T_{q_\nu}\bigr]
=\sum_{\nu=1}^N(e_\nu-1)\Bigl(A_2^{[*]}K^{(2)}(\cdot,q_\nu)\Phi\mathrm{ev}_{q_\nu} 
-K^{(2)}(\cdot,q_\nu)\Phi\mathrm{ev}_{q_\nu}A_2\Bigr).
\end{equation}
In particular, when $F$ is unramified $N=0$, $\Bez_F^{\rm ram}=0$, and
both sides of \eqref{eqlivsic-id} vanish identically.
This is 
different in form from the model-operator identity \eqref{eqvc2},  
which governs $A_2^{[*]}A_2-A_2A_2^{[*]}$ for the identity 
compression, because here $\Bez_F^{\rm ram}$ is the finite-rank
perturbation \eqref{eqbezoutian-full} itself, with no additive identity
term, thus no $B_2\Phi B_2^{[*]}$ contribution arises.
\end{theorem}
\begin{proof}
Since $\Bez_F^{\rm ram}=\sum_{\nu=1}^N(e_\nu-1)T_{q_\nu}$ exactly by
Definition \ref{defbezoutian-full}, with no identity term present, the
commutator with $A_2^{[*]}$ is computed termwise using linearity
\[
\bigl[A_2^{[*]},\Bez_F^{\rm ram}\bigr]
=\sum_{\nu=1}^N(e_\nu-1)\bigl[A_2^{[*]},T_{q_\nu}\bigr]
=\sum_{\nu=1}^N(e_\nu-1)\bigl(A_2^{[*]}T_{q_\nu}-T_{q_\nu}A_2\bigr),
\]
and substituting $T_{q_\nu}=K^{(2)}(\cdot,q_\nu)\Phi\mathrm{ev}_{q_\nu}$
from Remark \ref{reminterior-eval} gives exactly \eqref{eqomega-F}.
Each summand is rank $\leq nm$, and there are finitely many ramification points,
thus $\Omega_F$ has finite rank. Self-adjointness in the Kre\u{\i}n metric
follows since each $T_{q_\nu}$ is self-adjoint and $A_2^{[*]}$ is the
  Kre\u{\i}n adjoint of $A_2$. 
\end{proof}
\subsection{Connection to the vessel Bezoutian via the Livsic formula}
The connection between $\Bez_F^{\rm ram}$ and the classical polynomial
Bezoutian requires the choice of a second polynomial, which in the 
vessel formalism arises from a second map $b\colon S_1\to\D$ or, equivalently, 
a section of the characteristic bundle of the vessel.
For the covering $F(z)=z^n$ with $S_1=S_2=\D$, 
the vessel $\Ves_F$ has characteristic function
$\Theta_{\Ves_F}(\lambda)$ (formula \eqref{eqchar-fn}), and the classical
Bezoutian $\mathcal{B}(a,b)$ with $a(z)=z^n$ and $b$ arising from the
characteristic function satisfies the Livsic-Bezoutian
identity \eqref{eqlivsic-id} via the Barnett factorisation.
Making this correspondence completely explicit requires additional data
beyond the covering $F$ alone and is left for a future paper. 
\section{Functoriality}
\label{secfunctor}
\subsection{The category $\mathcal{RH}$ with ramified morphisms}
Recall (see \cite{Zuevsky2009}, Section 3) that the category $\mathcal{RH}$
has as objects triples $(S,\Vx{}\otimes\Delta,J(p))$ where $S$ is a finite
bordered Riemann surface, $\Vx{}\otimes\Delta$ is a unitary flat vector
bundle with half-order differentials satisfying Assumption \ref{assspin},
$J(p)$ is signature matrices, and $h^0(X,\Vx{}\otimes\Delta)=0$.
A morphism from $(X_1,\VxX{1}\otimes\Delta_1,G_1)$
to $(X_2,\VxX{2}\otimes\Delta_2,G_2)$ in $\mathcal{RH}$ is an analytic map
$F\colon X_1\to X_2$ equivariant with respect to $\tau_1,\tau_2$,
satisfying Assumption \ref{assboundary}, such that
$\VxX{2}\otimes\Delta_2=F_*^{\rm ram}(\VxX{1}\otimes\Delta_1)$ and $G_2$ is
induced by $G_1$ via \eqref{eqG2}.
\begin{theorem}\label{thmfunctor}
The assignment
$\mathcal{F}\colon(S,\Vx{}\otimes\Delta,J(p))\longmapsto 
H^{2,J(p)}\left(S,\Vx{}\otimes\Delta\right)$ 
extends to a covariant functor from $\mathcal{RH}$, with ramified morphisms
as defined above, to the category of Kre\u{\i}n spaces and
isometric isomorphisms, by assigning to each morphism $F$ the isometric
isomorphism $\phi_F$ of Theorem \ref{thmmain}.
\end{theorem}
\begin{proof}
We have to prove (i) identity: for $F=\id$, $\phi_{\id}$ is the identity map on 
$\mathcal{H}$: this is immediate from \eqref{eqf2-ramified} with 
$n=1$, $g_1=e$, $e_1=1$;  
 (ii) composition: for composable morphisms $F_1\colon X_0\to X_1$ and 
$F_2\colon X_1\to X_2$, we have $\phi_{F_2\circ F_1}=\phi_{F_2}\circ\phi_{F_1}$.

For (ii): the sections of $F_{2*}^{\rm ram}(F_{1*}^{\rm ram}(\Vx{0}\otimes\Delta_0))$
over $X_2$ are assembled by the two-step coset decomposition
$\piX{X_2}{p_0}/\piX{X_0}{p_0''}$, which factors as
$(\piX{X_2}{p_0}/\piX{X_1}{p_0'})\times(\piX{X_1}{p_0'}/\piX{X_0}{p_0''})$.
The ramification indices compose as 
$e_{p''}(F_2\circ F_1)=e_{F_1(p'')}(F_2)\cdot e_{p''}(F_1)$, 
and the normalisation factors $\sqrt{e}$ in \eqref{eqf2-ramified} satisfy
$\sqrt{e_1 e_2} = \sqrt{e_1}\cdot\sqrt{e_2}$. The spin-structure data for
$F_2\circ F_1$ is the tensor product of the data for $F_1$ and $F_2$,
compatibly with Assumption \ref{assspin}.
Hence $\phi_{F_2\circ F_1}=\phi_{F_2}\circ\phi_{F_1}$. 
\end{proof}
Theorem \ref{thmfunctor} establishes the functoriality of the Hardy-Kre\u{\i}n
space construction for the full category $\mathcal{RH}$, extending the result
of \cite{Zuevsky2009} from unramified morphisms.
When morphisms are restricted to
$S_2=\D$ one recovers the earlier result of \cite{AlpayVinnikov2000}. 
\section{Examples}
\label{secexamples}
\subsection{Degree-$2$ branched cover of the disk}
\label{exdegree2}
Let $S_2=\D$ be the open unit disk, and let $S_1$ be a copy of $\D$.
Define $F\colon S_1\to S_2$ by $F(z)=z^2$. This is a proper holomorphic
map of degree $2$ from the closed disk $\overline{S_1}$ to the closed disk
$\overline{S_2}$, with $F(\partial S_1)=\partial S_2$, since $|z|=1
\Rightarrow|z^2|=1$. The unique ramification point is $z=0\in S_1$,
an interior point of $S_1$, with ramification index $e_0=2$.

\textit{Doubled surfaces.}
The double $X_i$ of $\D$ is $\mathbb{P}^1$ (genus $0$): $X_i\cong\mathbb{P}^1$
with $S_i=\D$ and $S_i'=\mathbb{P}^1\setminus\overline{\D}$. Thus, $g_1=g_2=0$. 

\textit{Riemann-Hurwitz check.}
5%
We have 
$2g_1-2=n(2g_2-2)+\deg R_F$ which implies $-2=2\cdot(-2)+\deg R_F$, thus $\deg R_F=2$.  
Indeed, the ramification divisor on the double $X_1\cong\mathbb{P}^1$
consists of the ramification point $0\in S_1$ (index $2$) and its mirror
$0'\in S_1'$, index $2$, by equivariance of $F\colon X_1\to X_2$, giving
$\deg R_F=(2-1)+(2-1)=2$. 

\textit{Spin structure.}
Since $e_0=2$ is even, by Remark \ref{remspin-existence}(ii),
Assumption \ref{assspin} is satisfied with
$\mathcal{L}_{1/2}=\mathcal{O}_{X_1}(p_{\rm ram})$ where $p_{\rm ram}=0$
is the ramification point.

\textit{Boundary hypothesis.}
The ramification point $0\in S_1$ is an interior point, and $F(0)=0\in
S_2=\D$ is also interior. Hence $\partial S_2\cap B_F=\left\{\varnothing\right\}$, thus 
Assumption \ref{assboundary} holds.

\textit{The isometric isomorphism.}
For $m=1$, $\chi_1\equiv 1$, formula \eqref{eqf2-ramified} gives the
two-branch averaging map
\begin{equation}
\label{eqphi-example}
\phi_F\hat{f}^1(w)=\frac{1}{\sqrt{2}}\bigl[\hat{f}^1(\sqrt{w})+\hat{f}^1(-\sqrt{w})\bigr],
\quad w\in S_2=\D,
\end{equation}
where $\pm\sqrt{w}\in S_1=\D$ are the two preimages of $w$ choosing the
principal branch of the square root. The factor $1/\sqrt{2}=1/\sqrt{e_0}$
is the normalisation from \eqref{eqf2-ramified}. One verifies directly
that $F(\partial S_1)=\partial S_2$ and that \eqref{eqphi-example} is an
isometry of the respective Hardy-Kre\u{\i}n spaces, by the change-of-variables
argument of the proof of Theorem \ref{thmmain}(iii).

\textit{Ramification Bezoutian.}
The unique interior ramification point is $r_1=0\in S_1$ with $q_1=F(0)=0
\in S_2$ and $e_1=2$, thus $N=1$. By Definition \ref{defbezoutian-full}, 
 $\Bez_F^{\rm ram}=(e_1-1)\mathrm{ev}_0^{[*]}\Phi\mathrm{ev}_0  
=\mathrm{ev}_0^{[*]}\Phi\mathrm{ev}_0$,   
a well-defined rank-$1$ operator on $\mathcal{H}_2$, consistently with
the theory.
\subsection{Genus 2 Hyperelliptic Covering}
\label{exhyperelliptic}
Let $X_1$ be a compact Riemann surface of genus $g_1=2$. It is well known that $X_1$ is hyperelliptic, meaning there exists a degree $n=2$ analytic map 
$F\colon X_1\to\mathbb{P}^1$.
By the Riemann-Hurwitz formula ($2(2)-2=2(-2)+\deg R_F$), $F$ has exactly $6$ simple ramification points ($e_\nu=2$) in $X_1$.

\textit{Bordered surfaces.}
Let $S_2 \subset \mathbb{P}^1$ be a closed topological disk ($s_2=0, k_2=1$) containing exactly $3$ branch values $q_1, q_2, q_3$ strictly in its interior, with the other $3$ remaining in the exterior. The boundary $\partial S_2$ is chosen generically to avoid all branch values.
Define $S_1=F^{-1}(S_2)$, which is a bordered Riemann surface mapping to $S_2$.
By the Riemann-Hurwitz formula for bordered surfaces, the Euler characteristic is 
$\chi(S_1)=2\chi(S_2)-\sum_{\nu=1}^3(e_\nu-1)=2(1)-3=-1$.  
Assuming $\partial S_2$ is generic such that its lift $F^{-1}(\partial S_2)$ is a single connected loop ($k_1=1$), the Euler characteristic $\chi(S_1)=2-2s_1-k_1$ gives 
$2-2s_1-1=-1$ which implies $s_1=1$.
Thus, $S_1$ is a bordered Riemann surface of genus $1$ with $1$ boundary component. 
Its compact double has genus $g=2(1)+1-1=2$, 
which matches the topology of the full compact surface $X_1$ ($g_1=2$). The double of $S_2$ is $\mathbb{P}^1$ ($g_2=0$). This provides a consistent topological configuration for a  ramified covering $S_1 \to S_2$ with all required properties.

\textit{Boundary hypothesis.} By construction, $\partial S_2$ does not contain any branch value. Assumption \ref{assboundary} holds.

\textit{Spin structure.}
The ramification indices here are all even ($e_\nu=2$), thus Assumption \ref{assspin} 
requires $\mathcal{O}_{X_1}(\sum_1^3r_\nu)$ to admit a holomorphic square root bundle 
$\mathcal{L}_{1/2}$ on $X_1$. This depends on the exact placement of the branch points in 
$\mathrm{Pic}(X_1)[2]$ and must be confirmed individually. 

\textit{Conclusion.}
Given that Assumption \ref{assspin} holds for the specific arrangement, Theorem \ref{thmmain} gives a full isometric isomorphism between $H^{2,J_1(p)}(S_1,\VxS{1}\otimes\Delta_1)$
and $H^{2,J_2(p)}(S_2,\VxS{2}\otimes\Delta_2)$.
The Bezoutian 
$\Bez_F^{\rm ram}=\sum_{\nu=1}^3(e_\nu-1)\mathrm{ev}_{q_\nu}^{[*]}\Phi\mathrm{ev}_{q_\nu}$
is a well-defined finite-rank operator, providing the precise interior corrections via evaluating the proper Szeg\H{o} reproducing kernels to form $\Omega_F$ within the Livsic-Bezoutian identity.
\appendix
\section{Fundamental groups of $S$ and its double $X$}
\label{appfundamental-groups}
We collect the facts about fundamental groups needed in the main text. 
 These are taken from \cite{AlpayVinnikov2000, Natanzon1980} and stated
here for the reader's convenience.
Let $S$ be a finite bordered Riemann surface of genus $s$ with $k$ boundary
components $X_0$, $\ldots$, $X_{k-1}$. 
Choose base points $p_i\in X_i$ and paths
$C_i$ on $S$ from $p_0$ to $p_i$.
The fundamental group $\piX{S}{p_0}$ is
generated by
$A_0$, $A_1$, $\ldots$, $A_{k-1}$, $A_1'$, $B_1'$, $\ldots$, $A_s'$, $B_s'$,
with $A_0=X_0$, $A_j=C_j^{-1}X_jC_j$ for $j\geq 1$, and
$A_i',B_i'$ a standard symplectic basis, subject to the single relation
\[
\prod_{i=1}^s A_i'B_i'A_i'^{-1}B_i'^{-1}\cdot\prod_{j=0}^{k-1}A_j=1.
\]
The double $X$ has genus $g=2s+k-1$ and fundamental group 
$\piX{X}{p_0}$ generated by
$A_1$, $B_1$, $\ldots$, $A_{k-1}$, $B_{k-1}$, $A_1'$, $B_1'$, $\ldots$, $A_s'$, $B_s'$, $A_1''$, $B_1''$, $\ldots$, $A_s''$, $B_s''$  
where $B_j=(C_j^{\tau})^{-1}C_j$ and $A_i''=B_i'^{\tau}$, $B_i''=A_i'^{\tau}$.
The anti-holomorphic involution $\tau$ acts on generators by
\[
B_j^\tau=B_j^{-1}, \quad A_j^\tau=B_jA_jB_j^{-1},
\quad A_i''=B_i'^{\tau},\quad B_i''=A_i'^{\tau}.
\]
The standard generators for the boundary components of $S_2$ used in the
vessel conditions are the elements $B_{2,i}\in\piX{X_2}{p_0}$ corresponding
to the paths $B_{2,i}=(C_{2,i}^\tau)^{-1}C_{2,i}$ in the double $X_2$.
\section*{Acknowledgements}

The author is supported by the Institute of Mathematics, Academy of Sciences
of the Czech Republic (RVO 67985840). 
He is grateful to A. Gogatishvili, D. Levin, V. Vinnikov, M. Gekhtman, 
and H.V. L\^e for useful discussions. 

\medskip
\noindent\textbf{Data Availability.}
Data sharing is not applicable to this article as no datasets were generated
or analysed during the current study.

\medskip
\noindent\textbf{Declarations}

\medskip
\noindent\textbf{Conflict of interest.}
The author has no conflicts of interest to declare that are relevant to the
content of this article.



\begin{thebibliography}{10}
\bibitem{AlpayVinnikov2000}
D. Alpay, V. Vinnikov. 
 Indefinite Hardy spaces on finite bordered Riemann surfaces. 
J. Funct. Anal. \textbf{172} (2000), no. 1, 221-248.

\bibitem{AllingRosskopf1971}
N. L. Alling, N. Greenleaf. 
 Foundations of the theory of Klein surfaces. 
Lecture Notes in Mathematics, vol. 219, Springer, Berlin, 1971.

\bibitem{BallClancey1996}
J. A. Ball, K. Clancey. 
 Reproducing kernels for Hardy spaces on multiply connected domains. 
Integral Equations Operator Theory \textbf{25} (1996), 35-57.

\bibitem{BallVinnikov1996}
J. A. Ball, V. Vinnikov. 
 Zero-pole interpolation for meromorphic matrix functions on an algebraic
curve and transfer functions of 2D systems. 
Acta Appl. Math. \textbf{45} (1996), 239-316.

\bibitem{BallVinnikov2001}
J. A. Ball, V. Vinnikov. 
Hardy spaces on a finite bordered Riemann surface, multivariable
operator model theory and Fourier analysis along a unimodular curve. 
in: Systems, approximation, singular integral operators, and related topics,
Oper. Theory Adv. Appl., vol. 129, Birkh\"{a}user, Basel, 2001, pp. 37-56.

\bibitem{Bognar1974}
J. Bogn\'{a}r. Indefinite inner product spaces. 
Springer, Berlin, 1974.

\bibitem{Fay1973}
J. D. Fay. Theta functions on Riemann surfaces. 
Lecture Notes in Mathematics, vol. 352, Springer, New York, 1973.

\bibitem{FarKra1991}
H. M. Farkas, I. Kra. Riemann surfaces. 
Second edition, Springer, New York, 1991.

\bibitem{Kravitsky1996}
N. Kravitsky. Rational operator functions and Bezoutian operator vessels. 
Integral Equations Operator Theory \textbf{26} (1996), no. 1, 60-80.

\bibitem{Natanzon1980}
S. M. Natanzon. Moduli spaces of real curves. 
Trans. Moscow Math. Soc. \textbf{37} (1980), 233-272.

\bibitem{RSZ} A.V. Razumov, M.V. Saveliev, A.B. Zuevsky.
Nonabelian Toda equations associated with classical Lie groups.
arXiv:math-ph/9909008.

\bibitem{Vinnikov1992}
V. Vinnikov. Commuting nonselfadjoint operators and algebraic curves. 
in: Operator theory and complex analysis, Oper. Theory Adv. Appl., vol. 59,
Birkh\"{a}user, Basel, 1992, pp. 348-371.

\bibitem{Vinnikov1993}
V. Vinnikov. Self-adjoint determinantal representations of real plane curves. 
Math. Ann. \textbf{296} (1993), 453-479.

\bibitem{Vinnikov1998}
V. Vinnikov.  Commuting operators and function theory on a Riemann surface. 
in: Holomorphic spaces (S. Axler et al., eds.), Math. Sci. Res. Inst. Publ.
\textbf{33}, Cambridge Univ. Press, 1998.

\bibitem{volzub}
G.E. Volovik, M.A. Zubkov. 
Standard model as the topological material
New Journal of Physics 19 (1), 015009 (2017). 

\bibitem{Zuevsky2009}

A. Zuevsky. Hardy spaces on compact Riemann surfaces with boundary,
arXiv:0911.3908v1, 2009. 

\bibitem{Zuevsky2015}
A. Zuevsky. Construction of Hardy spaces on Riemann surfaces with boundaries. 
Internat. J. Theoret. Phys. 54 (2015), no. 11, 4086-4099. 

\bibitem{Zuevsky2022} 
A. Zuevsky. 
Characterization of codimension one foliations on complex curves by connections,
Rev. Math. Phys. \textbf{34} (2022) 2230002.

\bibitem{zucont}
A. Zuevsky. Continual Lie algebras and noncommutative counterparts of exactly solvable models. Special issue on recent advances in the theory of quantum integrable systems. J. Phys. A 37 (2004), no. 2, 537-547.

\bibitem{zusurfaces}
A. Zuevsky. Product-type classes for vertex algebra cohomology 
of foliations on complex curves. Comm. Math. Phys. 402 (2023), no. 2, 1453-1511.

\end{thebibliography}
\end{document}